\documentclass[12pt]{amsart}

\usepackage{mathrsfs}
\usepackage{txfonts}
\usepackage{amssymb,dsfont}
\usepackage{enumerate}
\usepackage{slashed}
\usepackage{fancyhdr}
\usepackage{aliascnt}
\usepackage[backref,pagebackref,breaklinks,colorlinks,linkcolor=red,citecolor=green,filecolor=magenta,urlcolor=cyan]{hyperref}
\usepackage[square,super,sort,longnamesfirst]{natbib}
\usepackage[dvips,dvipsnames]{xcolor}
\usepackage[dvips]{graphicx}
\usepackage[all,cmtip]{xy}
\usepackage{pdfsync}
\makeatletter
\newcommand{\autorefcheckize}[1]{%
  \expandafter\let\csname @@\string#1\endcsname#1%
  \expandafter\DeclareRobustCommand\csname relax\string#1\endcsname[1]{%
    \csname @@\string#1\endcsname{##1}\wrtusdrf{##1}}%
  \expandafter\let\expandafter#1\csname relax\string#1\endcsname
}
\makeatother
\newcommand{\abs}[1]{\left\lvert#1\right\rvert}%absolute%
\newcommand{\set}[1]{\left\{#1\right\}}%set%
\newcommand{\hin}[2]{\left\langle#1,#2\right\rangle}%Hermite inner product%
\newcommand{\kh}[1]{\left(#1\right)}%kuohao%
\newcommand{\field}[1]{\mathbb{#1}}
\newcommand{\R}{\field{R}}%real number%
\newcommand{\Com}{\field{C}}%complex number%
\newcommand{\Lg}[1]{\mathrm{#1}}%Lie group%
\newcommand{\La}[1]{\mathfrak{#1}}%Lie algebra%
\newcommand{\dif}{\mathrm d}

\newcommand{\To}{\longrightarrow}
\newcommand{\Rmn}[1]{\uppercase\expandafter{\romannueral#1}}%uppercase roman number%

\DeclareMathOperator{\Id}{Id}

\theoremstyle{plain}
\newtheorem{known}{Theorem}

\newtheorem{theorem}{Theorem}[section]

\newaliascnt{lem}{theorem}
\newtheorem{lem}[lem]{Lemma}
 \aliascntresetthe{lem}

\newaliascnt{cor}{theorem}
\newtheorem{cor}[cor]{Corollary}
 \aliascntresetthe{cor}

\newaliascnt{prop}{theorem}

 \aliascntresetthe{prop}

\theoremstyle{remark}
\newtheorem{rem}{Remark}[section]

\theoremstyle{definition}

\newtheorem{eg}{Example}[section]

\numberwithin{equation}{section}

\begin{document}

%title%
\title[lower eigenvalue estimates]{Optimal lower eigenvalue estimates for Hodge-Laplacian and applications}
\thanks{This work is partially supported by National Natural Science Foundation
 of China (Grant No. 11601442) and
Fundamental Research Funds for the Central Universities (Grant No. 2682016CX114 ).}

\author{Qing Cui}
\address{School of Mathematics,
Southwest Jiaotong University, 611756
Chengdu, Sichuan, China}
\email{cuiqing@swjtu.edu.cn}

\author{Linlin Sun}
\address{School of Mathematics and Statistics, Wuhan University, Wuhan, 430072, China}
\address{Computational Science Hubei Key Laboratory, Wuhan University, Wuhan, 430072, China}
\email{sunll@whu.edu.cn}

\subjclass[2010]{58J50; 53C24; 53C40}

\keywords{Hodge Laplacian, eigenvalue estimate, rigidity theorem,
homology sphere theorem}

%abstract%
\begin{abstract}%

In this paper, we consider the eigenvalue problem for Hodge-Laplacian on
a Riemannian manifold $M$ isometrically immersed into another Riemannian manifold
$\bar M$ for arbitrary
codimension. We first assume the pull back Weitzenb\"{o}ck operator
(defined in Section 2) of
$\bar M$ bounded from below, and obtain an extrinsic lower bound
for the first eigenvalue of Hodge-Laplacian. As applications, we obtain some
rigidity results and a homology sphere theorem. Second, when the pull back  Weitzenb\"{o}ck operator
of $\bar M$ bounded from both sides, we give a lower bound of the first eigenvalue by
the Ricci curvature of $M$ and some extrinsic geometry.
As a consequence, we prove a weak Ejiri type theorem, that is, if the Ricci curvature
bounded from below pointwisely by a function of the norm square of the mean curvature vector, then
$M$ is a homology sphere.
In the end, we give an example to show that all the eigenvalue estimates and homology sphere theorems
are optimal when $\bar M$ has constant curvature.
\end{abstract}%

\maketitle

\section{Introduction}
Let $M^n$ be an $n$-dimensional Riemannian manifold. For each integer $0\le p\le n$, the Hodge-Laplacian (or Laplace–de Rham operator)
acting on $p$-forms of $M$ is defined by
$$
\Delta = d\delta +\delta d: \Omega^p(M)\to \Omega^p(M),
$$
where $d$ and $\delta$ are the differential and co-differential operator. Hodge-Laplacian is
an natural generalization (up to a sign) of Laplace-Beltrami operator acting on scalar functions (i.e., $0$-forms). For each
$0\le p\le n$, denote by $\lambda_{1,p}$ the first eigenvalue of Hodge-Laplacian, i.e.,
$$
\lambda_{1,p} = \inf_{0\neq\omega \in \Omega^p(M)}\set{\frac{\int_M\abs{d\omega}^2+\abs{\delta\omega}^2}{\int_M\abs{\omega}^2}}.
$$

Eigenvalue estimates for Laplace-Beltrami operator acting on scalar functions are intensively studied in a huge literature.
Compare to this, eigenvalue problems for Hodge-Laplaican attracted less attention, although
they also play an important role in revealing relations between geometry (curvature, etc.) and topology (cohomology, etc.) of
manifolds.
One of the difficulties of the eigenvalue estimates for Hodge-Laplacian
 is the algebraic complexity of $\Omega^p(M)$ (compare to $\Omega^0(M)$).

In recent years, a number of authors devoted to this problem (e.g. \cite{GueSav04, Kwo16, RauSav11, RauSav12, Sav09, Sav05,Sav14, Smo02}).
Among them, Guerini-Savo \cite{GueSav04}, Kwong \cite{Kwo16}, Raulot-Savo \cite{RauSav11, RauSav12} and  Savo \cite{Sav09}
 investigated eigenvalues for Hodge-Laplacian on a manifold with boundary;
 Savo  \cite{Sav05} and Smoczyk \cite{Smo02} studied eigenvalues for Hodge-Laplacian on submanifolds in Euclidean space or a sphere;
Raulot-Savo \cite{RauSav11} and Savo \cite{Sav14} also studied eigenvalues for Hodge-Laplacian on a hypersurface immersed into another
Riemannian manifold.

As we note from above, when the target manifold is not a space form, all the extrinsic results are of codimension one (be a hypersurface or  boundary of a Riemannian manifold).
It is natural to study eigenvalue problems of Hodge-Laplacian on a Riemannian manifold immersed into another with
arbitrary codimension.
To this end, in the present paper, we first give some optimal extrinsic lower eigenvalue estimates of Hodge-Laplacian
on a Riemannian manifold immersed into another with
arbitrary codimension. After that, as applications, we will prove some rigidity results, such as the homology sphere theorems.

Let $i: M^n\to \bar{M}^{n+m}$ be an isometric immersion from $n$-dimensional Riemannian manifold $M$ into
$(n+m)$-dimensional Riemannian manifold $\bar{M}$. Let $\nu\in T^\bot M$ be a unit normal vector,
and denote by $S_\nu$ the shape operator associated with $\nu$. Assume $\set{k_i}_{i=1}^n$ are the
principle curvatures of $S_\nu$. Denote by  $I_p$  the set of all $p$-multi-indices
$$
I_p = \set{\left.\set{j_1, \cdots, j_p}\right\vert 1\le j_1\le \cdots \le j_p\le n}.
$$
For a given $\alpha = \set{j_1, \cdots, j_p}\in I_p$, set $\alpha_\star = \set{1,\cdots, n}\backslash \alpha$,
and call
$$
K_\alpha = k_{j_1}+\cdots +k_{j_p}
$$
a
 {\it p-curvature} of $S_\nu$.
 Set
 \begin{equation*}
   \left\{\begin{aligned}
   \beta_p(x)=& \frac{1}{p(n-p)}\inf_{\alpha\in I_p} K_\alpha(x)K_{\alpha_\star}(x),\\
   \beta_p(M)=& \inf_{x\in M} \beta_p(x).
   \end{aligned}\right.
 \end{equation*}
%For each $1\le p\le n$, the Hodge-Laplacian $\Delta: \Lambda^p (M) \to \Lambda^p(M)$ is a self-adjoint
%endomorphism of $\Lambda^p(M)$, and assume $\lambda_{1,p}$ is the lowest eigenvalue of $\Delta$.
When
codimension $m=1$ and the curvature operator of $\bar{M}$ is bounded from below, Savo \cite[Theorem 7]{Sav14}
obtained an optimal extrinsic lower bound of $\lambda_{1,p}$:
%The aim of the present paper is two-fold. First, we wish to give an extrinsic eigenvalue estimate
%for Hodge-Laplacian operator on a Riemannian manifold $M$ isometrically immersed into a Riemannian
%manifold $\bar{M}$ with arbitrary codimension. As applications, we will prove some results including a rigidity theorem
%and a homology sphere theorem when the second fundamental form is bounded from above.
%Second, we are interested in finding the relation between the Ricci curvature and the topology of a submanifold
%immersed into a manifold which is more general than a space form.
%
%Eigenvalue estimate for Hodge-Laplacian plays an important role in topology and geometry.
%It has been being active since  its appearance and  plenty of
%works concerned on this topic are published. Most of these works are devoted to the  intrinsic eigenvalue
%estimate for Hodge-Laplacian, we refer the reader to a classical survey \cite{Cha84}.
%Recent years,  extrinsic eigenvalue estimate for Hodge-Laplacian attracts several authors' attention,
%such as (uncomplete list) Savo \cite{Sav05,Sav14}, Raulot and Savo
%\cite{RauSav11,RauSav12} and Kwong \cite{Kwo16}.
%In particular, Savo obtain the following eigenvalue estimate \cite[Theorem 7]{Sav14} when the curvature operator of the ambient manifold is
% bounded from below.

\begin{known}[Savo\cite{Sav14}]\label{TheoremA}
 Let $M^n$ be a closed hypersurface of $\bar{M}^{n+1}$, a manifold with
 curvature operator $\bar{\mathcal{R}}$ bounded from below by $c\in\mathbb R$. Let $ 1\le p \le \frac{n}{2}$. Then
$$
\lambda_{1,p}(M)\ge p(n-p+1)(c+\beta_p(M)),
$$
where   $\beta_p(M)$ is a constant defined above. If $M$ is a geodesic
sphere in a simply connected
manifold of constant curvature $c$, then equality holds.
\end{known}
\begin{rem}
  By Poincar\'{e} duality,
$\lambda_{1,p}(M)=\lambda_{1,n-p}(M)$, for all $0\le p \le n$. Moreover, $\lambda_{1,0}=\lambda_{1,n}=0$. Therefore, we always assume
$1\le p\le \frac{n}{2}$ if there is no other explanation.
\end{rem}
The main tool that Savo used to prove the above theorem is the Bochner formula, that is,
for all $\omega \in \Omega^p(M)$,
\begin{align*}
\dfrac12\Delta\abs{\omega}^2&=\abs{\nabla\omega}^2
-\hin{\Delta\omega}{\omega}
+\hin{W^{[p]}(\omega)}{\omega},
\end{align*}
where $W^{[p]}: \Omega^p(M)\to \Omega^p(M)$ is usually called {\it($p$-th) Weitzenb\"{o}ck operator}.
When $p=1$, $W^{[1]}$ is nothing but the Ricci tensor. But when $2\leq p\leq n-2$, $W^{[p]}$ is complicated and
 is hard to be controlled in a general case.
However, it is crucial and necessary to control the term $\hin{W^{[p]}\kh{\omega}}{\omega}$
in eigenvalue estimates or in
other problems. One can
also define the Weitzenb\"{o}ck operator $\bar{W}^{[p]}$ of $\bar M$. Denote by $i^*\bar{W}^{[p]}$ the
{\it pull-back Weitzenb\"{o}ck operator}, which is the restriction of $\bar{W}^{[p]}$ on $\Omega^p(M)$.
One can check  (c.f. \cite{GalMey75}) that, $\bar{\mathcal{R}}\ge c$ implies $i^*\bar{W}^{[p]}\ge p(n-p)c$.

%We generalize \autoref{TheoremA} in two directions. First, we consider the high codimension case.
%In this case, $\beta_p(M)$ can not be defined as above since there are more than one normal directions.
%To get the eigenvalue estimate, we need to deal with a more complicated
%algebraic inequality. Second, we release the hypothesis ``$\bar{\mathcal{R}} \ge c$"
%in \autoref{TheoremA}
%to the weaker assumption ``$i^*\bar{W}^{[p]}\ge p(n-p) c$".
%
For higher codimension, $\beta_p(M)$ can not be defined as above since there are more than one normal directions. However, we have
\begin{theorem}\label{eigenest}
Suppose $M^n$ is a closed submanifold in $\bar{M}^{n+m}$
with the pull back Weitzenb\"{o}ck operator
$i^*\bar{W}^{[p]}\geq p(n-p)c $ for some constant $c$ and $1\le p\le \frac{n}{2}$.
 Then
\begin{align*}
\lambda_{1,p}(M)\geq&p(n-p+1)
\left(c+\gamma_p\right),
\end{align*}
where $$\gamma_p=\min_{x\in M}\left\{\left(-\dfrac{1}{n}\abs{\mathring{B}}^2
-\dfrac{(n-2p)\abs{H}}{\sqrt{np(n-p)}}
\abs{\mathring{B}}+\abs{H}^2\right)(x)\right\},$$
 $H$ is the mean curvature vector and $\mathring{B}$ is the
traceless part of the second fundamental form $B$.
Moreover, if $M$ is totally umbilical, then
$$
\lambda_{1,p}(M)\ge p(n-p+1)\min_M\kh{c+\abs{H}^2}.
$$
\end{theorem}
The following corollary is a direct consequence of \autoref{eigenest}.
\begin{cor}\label{cor1}
  Assume the assumptions of \autoref{eigenest} hold, we also have the following
  eigenvalue estimates:
  \begin{itemize}
    \item[(i)] $\lambda_{1,p}(M) \ge p(n-p+1) \min_M\kh{c-\dfrac{n}{4p(n-p)}\abs{\mathring{B}}^2}$.
    \item[(ii)] $\lambda_{1,p}(M) \ge p(n-p+1) \min_M\kh{c-\dfrac{1}{2\sqrt{p(n-p)}}\abs{B}^2}$.
   \item[(iii)] If $n$ is even, then
    \begin{align*}
    \lambda_{1,n\slash 2}(M) \ge \dfrac{n(n+2)}{4}\min_{M}\kh{c
      -\dfrac{1}{n}\abs{\mathring{B}}^2 + \abs{H}^2}.
    \end{align*}
  \end{itemize}
\end{cor}
\begin{rem} (i)
It is worth pointing out that all the eigenvalue estimates in \autoref{eigenest} and
\autoref{cor1} are optimal. To see this, let $\bar M^{n+m}$  be a Riemannian manifold with
constant sectional curvature $c>0$, and $M^n$ be a geodesic sphere in $\bar M^{n+m}$. In this case,
$i^*W^{[p]}\equiv p(n-p) c$ and $M$ is totally umbilical. On the other hand, it is shown (c.f. \cite{GalMey75}) that
$\lambda_{1,p}(M) = p(n-p+1)\kh{c+\abs{H}^2}$. Thus, the eigenvalue estimates is optimal when
$M$ is umbilical. When $M$ is not umbilical,  the eigenvalue estimates are also optimal by
computing the first eigenvalue of Clifford torus (see the example in Appendix).  \\
(ii) When codiemsion $m=1$, Savo\cite{Sav14} obtained  \autoref{eigenest} and \autoref{cor1}. Moreover,  \autoref{eigenest} gives \autoref{TheoremA}.
\end{rem}

%It is easy to see the above eigenvalue estimate is sharp, since the case
% $M=\mathbb S^n$ and $\bar{M}=\mathbb R^{n+m}$ attains the equality.
%It is also worth pointing out that the condition  $i^*\bar{W}^{[p]}\ge c$
% is weaker
%than the condition that curvature operator is bounded from below by $c$.
%

As an application, we have the following rigidity result.
\begin{theorem}\label{rigidthm}
Suppose $M^n$ is a closed submanifold in $\bar M^{n+m}$,
$p\in\set{1,\dotsc, n-1}, i^*\bar W^{[p]}\geq c\geq0$ and
$\abs{B}^2\leq\alpha(c,p,n,H)$,
then the $p$-th betti number $b_p\leq\binom{n}{p}$,
where
\begin{align*}
\alpha(c,p,n,H)\coloneqq nc+\dfrac{n^3\abs{H}^2}{2p(n-p)}-\dfrac{n\abs{n-2p}\abs{H}\sqrt{n^2\abs{H}^2+4cp(n-p)}}{2p(n-p)}.
\end{align*}
Moreover, if $b_p>0$, then $\vert B\vert^2\equiv\alpha(c,p,n,H)$. In particular, if $\chi(M)\neq 1+(-1)^n$, then $\vert B\vert^2\equiv\alpha(c,p,n,H)$ for some $1\leq p\leq n-1$.
\end{theorem}
The following two corollaries are direct consequences of the above rigidity theorem.
\begin{cor}If the assumptions of \autoref{rigidthm} hold, and moreover if
 $\abs{B}^2\leq\alpha(c,1,n,H)$ and $b_p>0$ for some $1<p<n-1$, then either $\mathring{B}\equiv0$ or $M$ is minimal satisfying $\abs{B}^2\equiv nc$.
\end{cor}

\begin{cor}If the assumptions of \autoref{rigidthm} hold, and moreover if
 $n=0 \ \ (\mod4)$, $\abs{B}^2\leq\alpha(c,1,n,H)$ and the signature $sig(M)$ of $M$ is nonzero, then either $\mathring{B}\equiv0$ or $M$ is minimal satisfying $\abs{B}^2\equiv nc$.
\end{cor}

Moreover, if the strict inequality holds, we have a homology sphere theorem as follows.
\begin{theorem}\label{rationalsphere1}
If the assumptions of \autoref{rigidthm} hold, and moreover if
$\abs{B}^2<\alpha(c,1,n,H)$ and $M$ is simply connected,
then $M$ is a homology sphere, i.e.,
$H^i(M,\mathbb R)=0$ for all $1\leq i\leq n-1$.
\end{theorem}
\begin{rem}
About \autoref{rationalsphere1}, we should remark that,
\begin{itemize}
\item If $\bar M$ is a space form, then Shiohama-Xu\cite{ShiXu97} obtained a topological
sphere theorem  under the same condition, i.e., the condition $\abs{B}^2<\alpha(c,1,n,H)$
implies that $M$ is homeomorphic to a sphere.
\item Applying the rational Hurewicz theorem, in the conclusion, one actually have
 $\pi_i(M,\mathbb Q)=0$ for all $1\le i\le n-1$ (which can be called a {\it rational homotopy sphere}).
\end{itemize}
\end{rem}

It is worth noting that, constant $\gamma_p$ in \autoref{eigenest}
depends on $\abs{H}^2$ and $\abs{\mathring{B}}^2$. But by the Gauss equation \eqref{eq:gauss}, we have
\begin{align*}
  Scal_M=\sum_{i,j} \bar{R}_{ijij} + n(n-1)\abs{H}^2 -\abs{\mathring{B}}^2.
 \end{align*}
 Moreover, by the  definition of $i^*\bar W^{[p]}$,
$$\hin{i^*\bar W^{[1]} (\eta^i)}{\eta^i} =\sum_{j=1}^n\bar{R}_{ijij}.$$
We see that $\gamma_p$ actually depends on the scalar curvature of $M$, $\abs{H}^2$ and $ i^*\bar W^{[1]}$.
In this direction, assume $i^*\bar W^{[p]}$ bounded from below and $i^*\bar W^{[1]}$ bounded from above, we 
also obtain a lower eigenvalue estimate for Hodge-Laplacian
by the Ricci curvature of $M$ and $\abs{H}^2$. 

\begin{theorem}\label{thm:eigenricci}Suppose $M^n$ is a closed submanifold of $\bar{M}^{n+m}$ with $i^*\bar{W}^{[p]}\geq p(n-p)c_*$ and $i^*\bar{W}^{[1]}\leq (n-1)c^*$, where $c^*\ge c_*$ are two constants, then for $1\leq p\leq n/2 $,
\begin{align*}
&\dfrac{n-p}{n-p+1}\dfrac{n-2}{(n+2)p(n-p)-n^2}\lambda_{1,p}(M)\\
\geq&\left(Ric_{min}-(n-1)\left(c^*+\abs{H}^2\right)+\dfrac{(n-2)p(n-p)}{(n+2)p(n-p)-n^2}\left(c_*+\abs{H}^2\right)\right).
\end{align*}
\end{theorem}
As an application, we will give a new {\it Ejiri type homology sphere theorem}.
 There are several type of
sphere theorems depending on the curvature assumptions
 added on the submanifold or the target manifold. These curvature assumptions include
  pinched sectional curvature,
 bounded Ricci curvature, etc.
In Section 4, we will restrict our attention to the case that the Ricci curvature of the
submanifold is bounded from below.
As far as we know, the first such type result was given by Ejiri in 1979
for minimal submanifolds of a sphere.
Recently, Gu-Xu generalize Ejiri's result to submanifolds of space forms with parallel mean curvature vector.

\begin{known}[Ejiri\cite{Eji79}, Gu-Xu\cite{XuGu13}]
Let $M$ be an $n(\geq3)$-dimensional complete submanifold with parallel mean curvature vector $H$ in $F^{n+m}(c)$ with $c+\abs{H}^2>0$. If the Ricci curvature of $M$ satisfies
\begin{align*}
Ric_M\geq(n-2)\left(c+\abs{H}^2\right),
\end{align*}
then $M$ is either the totally geodesic submanifold $\mathbb{S}^n\left(\tfrac{1}{\sqrt{c+\abs{H}^2}}\right)$, the Clifford torus $\mathbb{S}^l\left(\tfrac{1}{\sqrt{2(c+\abs{H}^2)}}\right)\times\mathbb{S}^l\left(\tfrac{1}{\sqrt{2(c+\abs{H}^2)}}\right)$ in $\mathbb{S}^{n+1}\left(\tfrac{1}{\sqrt{c+\abs{H}^2}}\right)$ with $n=2l$, or $\Com P^2\left(\tfrac{4}{3}(c+\abs{H}^2)\right)$ in $\mathbb{S}^7\left(\tfrac{1}{c+\abs{H}^2}\right)$. Here $\Com P^2\left(\tfrac{4}{3}(c+\abs{H}^2)\right)$ denotes the $2$-dimensional complex projective space minimally immersed into $\mathbb{S}^7\left(\tfrac{1}{c+\abs{H}^2}\right)$ with constant holomorphic sectional curvature $\tfrac{4}{3}(1+\abs{H}^2)$.
\end{known}
Gu-Xu \cite{XuGu13} also obtain the following topological sphere theorem without the assumption of parallel mean curvature vector.
\begin{known}[Gu-Xu\cite{XuGu13}]\label{GXthm}Let $M$ be an $n$-dimensional closed submanifold with mean curvature vector $H$ in $F^{n+m}(c)$ with $c\geq0$. If the Ricci curvature of $M$ satisfies
\begin{align}\label{eq:condition-eriji}
Ric_M>(n-2)\left(c+\abs{H}^2\right),
\end{align}
then $M$ is homoemorphic to a sphere.
\end{known}

The original version of Gu-Xu's theorem assume that $n\geq4$. The case $n=2$ is a consequence of Gauss-Bonnet formula. The case $n=3$ is a consequence of Lawson-Simons theorem and Perelman's solution of Poincar\'{e} conjecture.

The key idea to prove \autoref{GXthm} is to claim that there is no
stable integral $p$-currents for $0<p<n$ under the assumption \eqref{eq:condition-eriji}. The {\it $p$-th weak Ricci curvature} of the $p$-plane $e_1\wedge e_2\wedge\dotsm\wedge e_p$ introduced by Gu-Xu\cite{GuXu12} is defined by
\begin{align*}
Ric(e_1\wedge e_2\wedge\dotsm\wedge e_p)\coloneqq\sum_{i=1}^pRic_{ii}.
\end{align*}
One can verify that $Ric(e_1\wedge e_2\wedge\dotsm\wedge e_p)$ is well defined, i.e., it is depending only on the $p$-plane $e_1\wedge e_2\wedge\dotsm\wedge e_p$. With an obvious modification of original results of Gu-Xu \cite{XuGu13} and Gu-Leng-Xu\cite{XuLenGu14}, one can obtain the following Theorem (for readers' convenience, we list a proof in Section 4).
\begin{known}\cite{XuLenGu14,XuGu13}\label{gclaim}
Let $M$ be an $n(\geq4)$-dimensional closed submanifold  with mean curvature vector $H$  in $F^{n+m}(c)$ with $c\geq0$.
If
\begin{align*}
\dfrac{Ric_{(p)}}{p}>\left(n-1-\dfrac{(n-2)p(n-p)}{(n+2)p(n-p)-n^2}\right)\left(c+\abs{H}^2\right),\quad 1<p<n-1,
\end{align*}
where $Ric_{(p)}$ is the lower bound of the $p$-th Ricci curvature, then there is no stable integral $p$-currents.
%In particular, if
%\begin{align*}
%Ric>(n-2)\left(c+\abs{H}^2\right)\geq0,
%\end{align*}
%then $M$ is homoemorphic to a sphere.
\end{known}

Note that the target manifold in the above three results are all of constant curvature.
In Section 4, as an application of \autoref{thm:eigenricci}, we generalize \autoref{GXthm} and  \autoref{gclaim} to
a more general case, that is, $\bar M$ is not necessarily of constant curvature.
\begin{theorem}\label{thm:eriji}Suppose $M^n$ is a closed submanifold of $\bar{M}^{n+m}$ with $i^*\bar{W}^{[p]}\geq p(n-p)c_*$ and $i^*\bar{W}^{[1]}\leq (n-1)c^*$, where $c^*\ge c_*$ are two constants. If
\begin{align*}
Ric>(n-1)\left(c^*+\abs{H}^2\right)-\dfrac{(n-2)p(n-p)}{(n+2)p(n-p)-n^2}\left(c_*+\abs{H}^2\right),
\end{align*}
holds for some $0<p<n$, then the $p$-th betti number is zero.
In particular,  if $M$ is simply connected, and
\begin{equation}\label{sharpric}
Ric>
\left\{\begin{aligned}
  &(n-1)\left(c^*+\abs{H}^2\right)-\left(c_*+\abs{H}^2\right),\quad n \ \ \text{is even},\\
  &(n-1)\left(c^*+\abs{H}^2\right)-\frac{(n-2)(n^2-1)}{n^3-2n^2-n-2}\kh{c_*+\abs{H}^2}, \quad n \ \ \text{is odd},
\end{aligned}\right.
\end{equation}
then $M$ is a homology sphere.
\end{theorem}
\begin{rem}
\begin{enumerate}
\item Suppose the sectional curvature of $\bar M$ is bounded below by $\bar K_{\min}$ and above by $\bar K_{\max}$, then we can take
\begin{align*}
c^*=&(n-1)\bar K_{\max},\\
c_*=&\dfrac{2[n/2]+1}{3}\left(\bar K_{\min}-\dfrac{2[n/2]-2}{2[n/2]+1}\bar K_{\max}\right).
\end{align*}
Therefore, our assumption is indeed weaker than constant curvature assumption.
\item The condition \eqref{sharpric} is sharp for all $p$ and all $n$
(no matter $n$ is even or odd) when $\bar M=F^{n+m}(c)$ (see the example in Appendix).
%{\color{blue} This answer a conjecture proposed by Gu-Xu.}

\end{enumerate}

\end{rem}
We call the above result a {\it weak Ejiri type theorem}. It is weak in the sense that
$M$ is just a homology sphere in our conclusion.
%because in the conclusion we only
%get $\pi_i(M, \mathbb Q)=0$ for all $1\le i\le n-1$.
Therefore, it can be seen as a generalization
of \autoref{GXthm} in the rational homotopy sense.
We emphasize that the proof of \autoref{thm:eriji}
bases on the Bochner's method which is quite
different from Ejiri's and Gu-Xu's.

The paper is organized as follows. In Section 2, we set up notation and terminology, and review some
of the standard facts on submanifolds geometry and  Hodge-Laplacian.
In Section 3, we give the proofs of \autoref{eigenest} and  \autoref{rigidthm}. We also give
another two applications of \autoref{eigenest}.
In Section 4, we give the proof of \autoref{thm:eigenricci} and  \autoref{thm:eriji}.
In the Appendix, we calculate an example of Clifford torus to show that the eigenvalue estimates
and sphere theorems are all optimal when the target manifold is the standard sphere.

{\bf Acknowledgements}: The authors would like to thank Professor Gu Juanru for helpful suggestions.

\section{Preliminaries}
In this section, we first recall some of the standard facts on submanifold geometry.

Let $i:M^n\to \bar{M}^{n+m}$ be an isometric immersion from a closed $n$-dimensional
Riemannian manifold $M$ to an $(n+m)$-dimensional Riemannian manifold $\bar{M}$. Let $e_1, \cdots, e_n, \nu_1, \cdots,\nu_m$ be an orthonormal frame on $\bar{M}$ such that $e_1, \cdots, e_n$ are tangent to $M$ and $\nu_1, \cdots,\nu_m$ are perpendicular to $M$, and $\eta^1, \cdots, \eta^n$ be the dual of $e_1, \cdots, e_n$. Let $R$ (resp. $\bar{R}$) be the (0,4)-type curvature tensor of $M$ (resp. $\bar{M}$), and $\mathcal{R}:\Lambda^2TM\to \Lambda^2TM$ be the curvature operator defined by
$$
\langle \mathcal{R}(e_i\wedge e_j),e_k\wedge e_l\rangle = R(e_i,e_j,e_k,e_l)\eqqcolon R_{ijkl}.
$$
From now on, we assume the Latin subscripts (or superscripts) $i,j,k,l,\cdots$  range from 1 to $n$, and the Greek subscripts (or superscripts) $\alpha,\beta, \gamma, \cdots$ range from 1 to $m$, and we will adopt the Einstein summation rule.
The second fundamental form and the mean curvature vector are given by
$$
B=h^\alpha_{ij} \eta^i\otimes\eta^j\otimes\nu_\alpha,\quad H=\frac{1}{n}\sum_i h_{ii}^\alpha\nu_\alpha\eqqcolon H^\alpha\nu_\alpha.
$$
and write $\mathring{B} = B - H\otimes g$ which is the traceless part of $B$, where $g$ is the metric on $M$.
Let $A$ be the shape operator defined by
$$
\hin{B(X,Y)}{\nu}= \hin{A^{\nu}(X)}{Y}, \ \  \text{for all }\ \  X,Y\in TM \ \ \text{and }\ \ \nu\in T^\bot M,
$$
Write $$
A^\alpha:= A^{\nu^\alpha}\quad \text{and} \quad \mathring{A}^\alpha := A^\alpha - g.$$
Recall the  Gauss equation
\begin{align}\label{eq:gauss}
R_{ijkl}=\bar R_{ijkl}+\sum_{\alpha=1}^m\left(h_{ik}^{\alpha}h_{jl}^{\alpha}-h_{il}^{\alpha}h_{jk}^{\alpha}\right).
\end{align}

Second, we summarize the relevant material on Hodge-Laplacian and some facts of its first eigenvalue.

Let $\Delta$ be the Hodge-Laplacian, i.e., $\Delta=\dif\delta+\delta\dif$. Since $[\dif,\Delta]=[\delta,\Delta]=0$, we have
\begin{align*}
\Delta:\dif\Omega^p(M)\To\dif\Omega^p(M),\\
\Delta:\delta\Omega^p(M)\To\delta\Omega^p(M).
\end{align*}
Let $\lambda_{1,p}^e$ and $\lambda_{1,p}^{ce}$ be the first eigenvalue of $\Delta$ acting on the exact and co-exact $p$-form on $M$ respectively. It is easy to check that the first eigenvalue $\lambda_{1,p}$ of $\Delta$ satisfying
\begin{align*}
\lambda_{1,p}\leq\min\set{\lambda_{1,p}^e,\lambda_{1,p}^{ce}}.
\end{align*}
By Hodge decomposition, we know that
\begin{align*}
\lambda_{1,p}=\min\set{\lambda_{1,p}^e,\lambda_{1,p}^{ce}},
\end{align*}
provided $H^p(M,\R)=0$.
 By Hodge duality,
\begin{align*}
\lambda_{1,p}^{e}=\lambda_{1,n-p}^{ce},\quad \lambda_{1,p}=\lambda_{1,n-p}.
\end{align*}
Thus, by differentiating eigenforms, we obtain that if $H^p(M,\R)=0$, then
\begin{align*}
\lambda_{1,p-1}^{ce}\leq\lambda_{1,p}^{e},\quad\lambda_{1,p+1}^{e}\leq\lambda_{1,p}^{ce}.
\end{align*}
In particular, if $H^{p-1}(M,\R)=H^{p}(M,\R)=0$, we have
\begin{align*}
\lambda_{1,p-1}^{ce}=\lambda_{1,p}^{e}.
\end{align*}
Moreover, if $H^{p}(M,\R)=0$, then
\begin{align*}
\min\set{\lambda_{1,p-1},\lambda_{1,p+1}}\leq\lambda_{1,p}.
\end{align*}
For example,
\begin{align*}
\lambda_{1,p}^{e}\left(\Lg{S}^n(1)\right)=p(n-p+1),\quad\lambda_{1,p}^{ce}\left(\Lg{S}^n(1)\right)=(p+1)(n-p).
\end{align*}

Third, we briefly sketch the Weitzenb\"{o}ck formula and Bochner formula for differential forms.

For every $p$-form $\omega$ on $M$, using the local orthonormal frame,
the Weitzenb\"{o}ck operator $W^{[p]}:\Omega^p(M)\to \Omega^p(M)$ is given by
(c.f. \cite{Jos11}),
$$
W^{[p]}(\omega)=\eta^i\wedge\iota_{e_j}R(e_i,e_j)\omega.
$$
Similarly, the pull back Weitzenb\"{o}ck operator $i^*\bar W^{[p]}:\Omega^p(M)\to \Omega^p(M)$  is given by
$$
i^*\bar W^{[p]}(\omega)=\eta^i\wedge\iota_{e_j}\bar R(e_i,e_j)\omega.
$$
The following two  formulas for $p$-forms are well known,
%
%\begin{lem}\label{lem:W}[c.f. \cite{Jos11}]For every $p$-form $\omega$ on $M$, we have
\begin{align}\label{Wformula}
\Delta\omega&=\nabla^\star\nabla\omega+W^{[p]}(\omega),\\
\label{Bformula}\dfrac12\Delta\abs{\omega}^2&=\abs{\nabla\omega}^2
-\hin{\Delta\omega}{\omega}
+\hin{W^{[p]}(\omega)}{\omega},
\end{align}
where   $\nabla^\star\nabla$ is the connection Laplacian.
Equalities (\ref{Wformula}) and (\ref{Bformula}) are usually called Weitzenb\"{o}ck formula and Bochner formula.
A direct computation gives (c.f. \cite{LawMic89})
\begin{align*}
\hin{W^{[p]}(\omega)}{\omega}=\dfrac14\hin{\mathcal{R}(\theta_I)}{\theta_J}\hin{\La{ad}_{\theta_I}\omega}{\La{ad}_{\theta_J}\omega},
\end{align*}
where $\{\theta_I\}$ is a local orthonormal frame of $\Lambda^2TM$.

In the end of this section, let us recall the following theorem due to Lawson-Simons \cite{LawSim73} ($c>0$) and Xin \cite{Xin84} ($c=0$).
\begin{known}[Lawson-Simons, Xin]Suppose $M^n\subset F^{n+m}(c), c\geq0$ and for every orthonormal frame $\set{e_i}$ of $TM$,
\begin{align}\label{eq:L-S}
\sum_{i=1}^p\sum_{j=p+1}^n\sum_{\alpha=1}^m\left(2\left(h_{ij}^{\alpha}\right)^2-h_{ii}^{\alpha}h_{jj}^{\alpha}\right)<p(n-p)c,
\end{align}
then there is  no stable integral $p$-currents, where $F^{n+m}(c)$ is the $(n+m)$-dimensional space form with sectional curvature $c$.
\end{known}

\section{Eigenvalue estimate and its applications}
In this section, we give the proofs of \autoref{eigenest} and \autoref{rigidthm} and
their applications.
%The proof of \autoref{eigenest} follows immediately from the following three
%lemmas.

To prove \autoref{eigenest}, we need the following lemma which is due to Gallot and Meyer \cite{GalMey75}, and we give a proof here for completeness.
\begin{lem}\label{lem:Wp} Let $M^n$ be an $n$-dimensional closed Riemannian manifold. For $1\le p\le \frac{n}{2}$, assume the
Weitzenb\"{o}ck operator
  $W^{[p]}\ge p(n-p)\Lambda$ for some $\Lambda >0$, then
  $$
  \lambda_{1,p}(M) \ge p(n-p+1)\Lambda.
  $$
\end{lem}
\begin{proof}
  Introduce the twistor operator $P$ on $M$ acting on $p$-form $\omega$ by
\begin{equation*}
P_X\omega\coloneqq\nabla_X\omega-\dfrac{1}{p+1}\iota_X\dif\omega+\dfrac{1}{n+1-p}X^{\flat}\wedge\delta\omega,
\end{equation*}
where $X^{\flat}$ is the dual $1$-form defined by $X^{\flat}(e_i)=\hin{X}{e_i}$. Then the following identity holds,
\begin{align*}
\abs{\nabla\omega}^2=\abs{P\omega}^2+\dfrac{1}{p+1}\abs{\dif\omega}^2+\dfrac{1}{n+1-p}\abs{\delta\omega}^2.
\end{align*}
Now applying Bochner formula (\ref{Bformula}), and by the assumption $W^{[p]}\ge p(n-p)\Lambda$, we obtain
\begin{align*}
&\dfrac{p}{p+1}\int_M\abs{\dif\omega}^2+\dfrac{n-p}{n+1-p}\int_M\abs{\delta\omega}^2\\
=&\int_M\abs{P\omega}^2+\hin{W^{[p]}(\omega)}{\omega}\\
\ge & p(n-p)\Lambda \int \abs{\omega}^2.
\end{align*}
By hypothesis $1\le p\le \frac{n}{2}$, we have $\frac{p}{p+1}\le \frac{n-p}{n+1-p}$. Therefore,

\begin{align*}
\dfrac{n-p}{n+1-p}\kh{\int_M\abs{\dif\omega}^2+\int_M\abs{\delta\omega}^2}
\ge  p(n-p)\Lambda \int \abs{\omega}^2.
\end{align*}
The conclusion follows from the variational characteristic  of the first eigenvalue.
\end{proof}

\begin{proof}[Proof of \autoref{eigenest}]
We will adopt the notations in Section 2, and for simplification, we introduce two more notations,
\begin{align*}
S(\omega)\coloneqq\eta^i\wedge\iota_{A^{\alpha}(e_i)}\omega\otimes\nu_{\alpha},
\quad
\mathring{S}(\omega)\coloneqq\eta^i\wedge\iota_{\mathring{A^{\alpha}}(e_i)}\omega\otimes\nu_{\alpha}.
\end{align*}
Direct calculations yield
\begin{align*}
W^{[p]}=&\eta^i\wedge\iota_{e_j}R(e_j,e_i)
=R_{ijkl}\eta^j\wedge\iota_{e_i}\left(\eta^k\wedge\iota_{e_l}\right),\\
i^*\bar W^{[p]}=&\eta^i\wedge\iota_{e_j}\bar R(e_j,e_i)
=\bar R_{ijkl}\eta^j\wedge\iota_{e_i}\left(\eta^k\wedge\iota_{e_l}\right).\\
\end{align*}
Hence, by using Gauss equation \eqref{eq:gauss},
\begin{align*}
&\hin{W^{[p]}(\omega)}{\omega}-\hin{i^*\bar W^{[p]}(\omega)}{\omega}\\
=&\left(R_{ijkl}-\bar R_{ijkl}\right)\hin{\eta^i\wedge\iota_{e_j}\omega}{\eta^k\wedge\iota_{e_l}\omega}\\
=&\sum_{\alpha}\left(h^{\alpha}_{ik}h^{\alpha}_{jl}-h^{\alpha}_{il}h^{\alpha}_{jk}\right)\hin{\eta^i\wedge\iota_{e_j}\omega}{\eta^k\wedge\iota_{e_l}\omega}\\
=&\sum_{\alpha}\hin{\eta^i\wedge\iota_{A^{\alpha}(e_l)}\omega}{A^{\alpha}(e_i)\wedge\iota_{e_l}\omega}-\sum_{\alpha}\hin{\eta^i\wedge\iota_{A^{\alpha}(e_k)}\omega}{\eta^k\wedge\iota_{A^{\alpha}(e_i)}\omega}\\
=&\sum_{\alpha}\hin{\iota_{A^{\alpha}(e_l)}\omega}{nH^{\alpha}\iota_{e_l}\omega+A^{\alpha}(e_i)\wedge\iota_{e_l}\iota_{e_i}\omega}-\sum_{\alpha}\hin{\iota_{A^{\alpha}(e_k)}\omega}{\iota_{A^{\alpha}(e_k)}\omega-\eta^k\wedge\iota_{e_i}\iota_{A^{\alpha}(e_i)}\omega}\\
=&\sum_{\alpha}\hin{\eta^l\wedge\iota_{nH^{\alpha}A^{\alpha}(e_l)}\omega}{\omega}-\sum_{\alpha}\abs{\sum_l\eta^l\wedge\iota_{A^{\alpha}(e_l)}\omega}^2+\sum_{\alpha}\abs{\sum_{k}\iota_{e_k}\iota_{A^{\alpha}(e_k)}\omega}^2\\
=&\sum_{\alpha}\hin{\eta^l\wedge\iota_{A^{\alpha}(e_l)}\omega}{nH^{\alpha}\omega}-\sum_{\alpha}\abs{\sum_l\eta^l\wedge\iota_{A^{\alpha}(e_l)}\omega}^2\\
=&-\sum_{\alpha}\abs{\sum_l\eta^l\wedge\iota_{A^{\alpha}(e_l)}\omega-\dfrac{n}{2}H^{\alpha}\omega}^2+\dfrac{n^2}{4}\abs{H}^2\abs{\omega}^2\\
=&-\abs{S(\omega)-\dfrac{n}{2}H\omega}^2+\dfrac{n^2}{4}\abs{H}^2\abs{\omega}^2.
\end{align*}
Therefore, by the assumption of the theorem, we have
\begin{align*}
  \hin{W^{[p]}(\omega)}{\omega}=&\hin{i^*\bar W^{[p]}(\omega)}{\omega}-\abs{S(\omega)-\dfrac{n}{2}H\omega}^2+\dfrac{n^2}{4}\abs{H}^2\abs{\omega}^2\\
=&\hin{i^*\bar W^{[p]}(\omega)}{\omega}-\abs{\mathring{S}(\omega)-\dfrac{n-2p}{2}H\omega}^2+\dfrac{n^2}{4}\abs{H}^2\abs{\omega}^2\\
=&\hin{i^*\bar W^{[p]}(\omega)}{\omega}-\abs{\mathring{S}(\omega)}^2+(n-2p)\hin{\mathring{S}(\omega)}{H\omega}+p(n-p)\abs{H}^2\abs{\omega}^2\\
\ge & p(n-p)c\abs{\omega}^2-\abs{\mathring{S}(\omega)}^2+(n-2p)\hin{\mathring{S}(\omega)}{H\omega}+p(n-p)\abs{H}^2\abs{\omega}^2.
\end{align*}
Hence, according to \autoref{lem:Wp} and the above inequality, to prove the theorem, it is sufficient to prove that,
\begin{align*}
\abs{\mathring{S}\vert_{\Omega^p(M)}}^2_{op}\leq\dfrac{p(n-p)}{n}\abs{\mathring{B}}^2,
\end{align*}
where $\abs{\cdot}_{op}$ stands for the operator norm  when acting on $p$-forms.

By definition, acting on $p$-forms,
\begin{align*}
\mathring{S}(\omega)=\eta^i\wedge\iota_{\mathring{A}^{\alpha}(e_i)}\omega\otimes\nu_{\alpha}\eqqcolon\mathring{S}^{\alpha}(\omega)\otimes\nu_{\alpha},
\end{align*}
we have
\begin{align*}
\abs{\mathring{S}}_{op}^2=\left(\sup_{0\neq\omega\in\Omega^p(M)}\dfrac{\abs{\mathring{S}(\omega)}}{\abs{\omega}}\right)^2\leq\sum_{\alpha=1}^m\abs{\mathring{S}^{\alpha}}_{op}^2.
\end{align*}
On the other hand, $\abs{\mathring{B}}^2=\sum_{\alpha=1}^m\abs{\mathring{A}^{\alpha}}^2$. Hence, the remain proof can be reduced to codimension  $m=1$ case, which has already done (c. f. \cite{RauSav11}).
\end{proof}
Conclusion (iii) of \autoref{cor1} is a direct consequence of \autoref{eigenest} when $p=\frac{n}{2}$. Conclusions (i) and (ii) of \autoref{cor1} follow from the corollary below.
\begin{cor}
Assume the assumptions of \autoref{eigenest} hold, then for every $\varepsilon >-1$, we have
$$
\lambda_{1,p}(M)\ge p(n+1-p)\min_{M} \kh{c-\frac{1}{n(1+\varepsilon)}\kh{\frac{n^2}{4p(n-p)}+\varepsilon}\abs{\mathring{B}}^2 -\varepsilon \abs{H}^2}.
$$
\end{cor}
\begin{proof}
  Direct calculations by \autoref{eigenest} and Young inequality.
\end{proof}

%\begin{proof}[Proof of \autoref{eigenest}]
%Introduce the twistor operator $P$ on $M$ acting on $p$-form $\omega$ by
%\begin{equation*}
%P_X\omega\coloneqq\nabla_X\omega-\dfrac{1}{p+1}\iota_X\dif\omega+\dfrac{1}{n+1-p}X^{\flat}\wedge\delta\omega,
%\end{equation*}
%where $X^{\flat}$ is the dual $1$-form defined by $X^{\flat}(e_i)=\hin{X}{e_i}$. Then the following identity holds,
%\begin{align*}
%\abs{\nabla\omega}^2=\abs{P\omega}^2+\dfrac{1}{p+1}\abs{\dif\omega}^2+\dfrac{1}{n+1-p}\abs{\delta\omega}^2.
%\end{align*}
%Now applying Bochner formula (\ref{Bformula}), we obtain
%\begin{align*}
%\dfrac{p}{p+1}\int_M\abs{\dif\omega}^2+\dfrac{n-p}{n+1-p}\int_M\abs{\delta\omega}^2=\int_M\abs{P\omega}^2+\hin{W^{[p]}(\omega)}{\omega}.
%\end{align*}
%
%Suppose $W^{[p]}\geq p(n-p)f$, then by variational characteristic of the first eigenvalue, we have
%\begin{align*}
%\lambda_{1,p}\geq&p(n+1-p)\times\min f.
%\end{align*}
%According to \autoref{lem:W3}, we know that
%\begin{align*}
%f\geq c+\dfrac{n^2\abs{H}^2}{4p(n-p)}-\dfrac{1}{p(n-p)}\abs{\mathring{S}\vert_{\Lambda^pT^*M}-\dfrac{n-2p}{2}H}^2_{op}.
%\end{align*}
%\end{proof}

\begin{proof}[Proof of \autoref{rigidthm}]
Notice that
\begin{align*}
c\geq\max\set{\dfrac{1}{n}\abs{\mathring{B}}^2+\dfrac{\abs{n-2p}\abs{H}}{\sqrt{np(n-p)}}\abs{\mathring{B}}-\abs{H}^2}
\end{align*}
is equivalent to
\begin{align*}
\abs{B}^2\leq nc+\dfrac{n^3\abs{H}^2}{2p(n-p)}-\dfrac{n\abs{n-2p}\abs{H}\sqrt{n^2\abs{H}^2+4cp(n-p)}}{2p(n-p)}.
\end{align*}
Hence, by assumption,  $\lambda_{1,p}(M)\geq0$. Thus, every harmonic $p$-form is a conformal killing form, i.e., $P\omega=0$, and is parallel. Moreover, if for some point, the strictly inequality holds, then there is no nontrivial harmonic $p$-form. In other words, $H^p(M,\R)=0$.

Notice that
\begin{align*}
&\min_{p\in\set{1,\dotsc,n-1}}\set{nc+\dfrac{n^3\abs{H}^2}{2p(n-p)}-\dfrac{n\abs{n-2p}\abs{H}\sqrt{n^2\abs{H}^2+4cp(n-p)}}{2p(n-p)}}\\
=&nc+\dfrac{n^3\abs{H}^2}{2(n-1)}-\dfrac{n(n-2)\abs{H}\sqrt{n^2\abs{H}^2+4c(n-1)}}{2(n-1)}.
\end{align*}
Consequently, if
\begin{align*}
\abs{B}^2\leq nc+\dfrac{n^3\abs{H}^2}{2(n-1)}-\dfrac{n(n-2)\abs{H}\sqrt{n^2\abs{H}^2+4c(n-1)}}{2(n-1)},
\end{align*}
then
\begin{align*}
b_p\coloneqq\dim_{\R}H^p(M,\R)\leq\binom{n}{p},\quad p\in\set{0,1,\dotsc, n}.
\end{align*}
Moreover, if the inequality holds strictly at some point, then
\begin{align*}
b_p\coloneqq\dim_{\R}H^p(M,\R)=0,\quad p\in\set{1,\dotsc, n-1}.
\end{align*}

Finally, if $\chi(M)\neq 1+(-1)^n$, there must be some $p\in\set{1,\dotsc, n-1}$ such that the betti number $b_p>0$. We finish the proof.
\end{proof}

\begin{proof}[Proof of \autoref{rationalsphere1}]Since $\abs{B}^2<\alpha(c,1,n,H)$, applying the estimate of the lower bound of the first $p$-eigenvalue, we know that the $p$-th betti number is zero for $0<p<n$, i.e., $M$ is a homology sphere.
%By using rational Hurewicz theorem, we know that $\pi_i(M,\Q)=0$ for all $1\leq i\leq n-1$.
\end{proof}

Besides the corollaries and theorems mentioned in the introduction, we have two more applications of  \autoref{eigenest}.
\begin{theorem}If the assumptions of \autoref{rigidthm} hold, and moreover if $\abs{\mathring{B}}^2\leq4cp(n-p)/n$ and the strictly inequality holds at some point, then $b_p=0$ for $p=1,\dotsc, n-1$. Therefore, if $\abs{\mathring{B}}^2<4c(n-1)/n$ and $M$ is simply connected, then $M$ is a rational homotopy sphere.
\end{theorem}
\begin{proof}
A direct computation gives
\begin{align*}
\min_{H}\set{\alpha(c,p,n,H)-n\abs{H}^2}=\dfrac{4cp(n-p)}{n},
\end{align*}
and the equality holds if and only if
\begin{align*}
n^2\abs{H}^2+4cp(n-p)=n^2c.
\end{align*}
Therefore, if $\abs{\mathring{B}}^2\leq\dfrac{4cp(n-p)}{n}$ and strictly inequality holds at some point, then $b_p=0$.
\end{proof}
Similarly,
\begin{theorem}If the assumptions of \autoref{rigidthm} hold, and moreover if
$\abs{B}^2\leq 2c\sqrt{p(n-p)}$ and the strictly inequality holds at some point, then $b_p=0$ for $p=1,\dotsc, n-1$. Therefore, if $\abs{B}^2<2c\sqrt{n-1}$ and $M$ is simply connected, then $M$ is a rational homotopy sphere.
\end{theorem}
\begin{proof}
Since
\begin{align*}
\min_{H}\alpha(c,p,n,H)=2c\sqrt{p(n-p)},
\end{align*}
we obtain the theorem.
\end{proof}

\section{Eigenvalue estimate and Ejiri's type Theorem}
In this section, we will prove \autoref{thm:eriji}.

First, for readers' convenience,  we provide here a different but simple proof of \autoref{gclaim}.
\begin{proof}[Proof of \autoref{gclaim}]
We need to verify Lawson-Simons condition \eqref{eq:L-S}.

We first study the case of $p=1$.
\begin{align*}
\sum_{i=1}^p\sum_{j=p+1}^n\sum_{\alpha=1}^m\left(2\left(h_{ij}^{\alpha}\right)^2-h_{ii}^{\alpha}h_{jj}^{\alpha}\right)=&\sum_{j=2}^n\sum_{\alpha=1}^m\left(h^{\alpha}_{1j}\right)^2-Ric_{11}+(n-1)c\\
\leq&\dfrac12\left(\abs{B}^2-n\abs{H}^2\right)-Ric_{11}+(n-1)c\\
=&\dfrac12\left(n(n-1)\left(c+\abs{H}^2\right)-\sum_{i=1}^nR_{ii}\right)-Ric_{11}+(n-1)c\\
=&\dfrac12n(n-1)\left(c+\abs{H}^2\right)-\dfrac12\sum_{i=1}^nR_{ii}-Ric_{11}+(n-1)c.
\end{align*}
Hence, if
\begin{align*}
Ric>\dfrac{n(n-1)}{n+2}\left(c+\abs{H}^2\right),
\end{align*}
then
\begin{align*}
\sum_{j=2}^n\sum_{\alpha=1}^m\left(2\left(h_{1j}^{\alpha}\right)^2-h_{11}^{\alpha}h_{jj}^{\alpha}\right)<(n-1)c,
\end{align*}
which means that there is no stable integral $1$-currents.

Now we consider the case of $2\leq p\leq n/2$.
\begin{align*}
&\sum_{i=1}^p\sum_{j=p+1}^n\sum_{\alpha=1}^m\left(2\left(h_{ij}^{\alpha}\right)^2-h_{ii}^{\alpha}h_{jj}^{\alpha}\right)\\
=&\sum_{i=1}^p\sum_{j=p+1}^n\sum_{\alpha=1}^m2\left(\mathring{h}_{ij}^{\alpha}\right)^2+\sum_{\alpha=1}^m\left(\left(\sum_{i=1}^p\mathring{h}^{\alpha}_{ii}\right)^2-(n-2p)H^{\alpha}\sum_{i=1}^p\mathring{h}^{\alpha}_{ii}\right)-p(n-p)\abs{H}^2\\
=&\sum_{i=1}^p\sum_{j=p+1}^n\sum_{\alpha=1}^m2\left(\mathring{h}_{ij}^{\alpha}\right)^2+\sum_{\alpha=1}^m\left(\left(\sum_{i=1}^p\mathring{h}^{\alpha}_{ii}-\dfrac{n-2}{2}pH^{\alpha}\right)^2+(p-1)nH^{\alpha}\sum_{i=1}^p\mathring{h}^{\alpha}_{ii}\right)\\
&-\dfrac{(n-2)^2}{4}p^2\abs{H}^2-p(n-p)\abs{H}^2.
\end{align*}
Similarly,
\begin{align*}
&\sum_{i=1}^p\sum_{j=p+1}^n\sum_{\alpha=1}^m\left(2\left(h_{ij}^{\alpha}\right)^2-h_{ii}^{\alpha}h_{jj}^{\alpha}\right)\\
=&\sum_{i=1}^p\sum_{j=p+1}^n\sum_{\alpha=1}^m2\left(\mathring{h}_{ij}^{\alpha}\right)^2+\sum_{\alpha=1}^m\left(\left(\sum_{j=p+1}^n\mathring{h}^{\alpha}_{jj}-\dfrac{n-2}{2}(n-p)H^{\alpha}\right)^2+(n-p-1)nH^{\alpha}\sum_{j=p+1}^n\mathring{h}^{\alpha}_{jj}\right)\\
&-\dfrac{(n-2)^2}{4}(n-p)^2\abs{H}^2-p(n-p)\abs{H}^2.
\end{align*}
Therefore,
\begin{align*}
&\sum_{i=1}^p\sum_{j=p+1}^n\sum_{\alpha=1}^m\left(2\left(h_{ij}^{\alpha}\right)^2-h_{ii}^{\alpha}h_{jj}^{\alpha}\right)\\
=&\dfrac{n-p-1}{n-2}\left(\sum_{i=1}^p\sum_{j=p+1}^n\sum_{\alpha=1}^m2\left(\mathring{h}_{ij}^{\alpha}\right)^2+\sum_{\alpha=1}^m\left(\sum_{i=1}^p\mathring{h}^{\alpha}_{ii}-\dfrac{n-2}{2}pH^{\alpha}\right)^2-\dfrac{(n-2)^2}{4}p^2\abs{H}^2\right)\\
&+\dfrac{p-1}{n-2}\left(\sum_{i=1}^p\sum_{j=p+1}^n\sum_{\alpha=1}^m2\left(\mathring{h}_{ij}^{\alpha}\right)^2+\sum_{\alpha=1}^m\left(\sum_{j=p+1}^n\mathring{h}^{\alpha}_{jj}-\dfrac{n-2}{2}(n-p)H^{\alpha}\right)^2-\dfrac{(n-2)^2}{4}(n-p)^2\abs{H}^2\right)\\
&-p(n-p)\abs{H}^2.
\end{align*}
 Thus, for $2\leq p\leq n/2$,
\begin{align*}
&\sum_{i=1}^p\sum_{j=p+1}^n\sum_{\alpha=1}^m\left(2\left(h_{ij}^{\alpha}\right)^2-h_{ii}^{\alpha}h_{jj}^{\alpha}\right)\\
\leq&\dfrac{n-p-1}{n-2}\left(\sum_{i=1}^p\sum_{j\neq i}\sum_{\alpha=1}^mp\left(\mathring{h}_{ij}^{\alpha}\right)^2+\sum_{\alpha=1}^mp\sum_{i=1}^p\left(\mathring{h}^{\alpha}_{ii}-\dfrac{n-2}{2}H^{\alpha}\right)^2-\dfrac{(n-2)^2}{4}p^2\abs{H}^2\right)\\
&+\dfrac{p-1}{n-2}\left(\sum_{j=p+1}^n\sum_{i\neq j}\sum_{\alpha=1}^m(n-p)\left(\mathring{h}_{ij}^{\alpha}\right)^2+\sum_{\alpha=1}^m(n-p)\sum_{j=p+1}^n\left(\mathring{h}^{\alpha}_{jj}-\dfrac{n-2}{2}H^{\alpha}\right)^2-\dfrac{(n-2)^2}{4}(n-p)^2\abs{H}^2\right)\\
&-p(n-p)\abs{H}^2\\
=&\dfrac{(n-p-1)p}{n-2}\sum_{i=1}^p\sum_{\alpha=1}^m\left(\sum_{j\neq i}\left(\mathring{h}_{ij}^{\alpha}\right)^2+\left(\mathring{h}^{\alpha}_{ii}-\dfrac{n-2}{2}H^{\alpha}\right)^2-\dfrac{n^2}{4}\abs{H^{\alpha}}^2+(n-1)\abs{H^{\alpha}}^2\right)\\
&+\dfrac{(p-1)(n-p)}{n-2}\sum_{j=p+1}^n\sum_{\alpha=1}^m\left(\sum_{i\neq j}\left(\mathring{h}_{ij}^{\alpha}\right)^2+\left(\mathring{h}^{\alpha}_{jj}-\dfrac{n-2}{2}H^{\alpha}\right)^2-\dfrac{n^2}{4}\abs{H^{\alpha}}^2+(n-1)\abs{H^{\alpha}}^2\right)\\
&-p(n-p)\abs{H}^2\\
\leq&\dfrac{(n-p-1)p}{n-2}(-K_p+(n-1)p\abs{H}^2)+\dfrac{(p-1)(n-p)}{n-2}\dfrac{n-p}{p}(-K_p+(n-1)p\abs{H}^2)-p(n-p)\abs{H}^2\\
=&-\dfrac{(n+2)p(n-p)-n^2}{(n-2)p}(K_p-(n-1)p\abs{H}^2)-p(n-p)\abs{H}^2,
\end{align*}
where
\begin{align*}
K_p=\min_{\set{e_i}}\set{\sum_{i=1}^pRic_{ii}-(n-1)pc}\eqqcolon Ric_{(p)}-(n-1)pc.
\end{align*}
Hence,
\begin{align*}
\sum_{i=1}^p\sum_{j=p+1}^n\sum_{\alpha=1}^m\left(2\left(h_{ij}^{\alpha}\right)^2-h_{ii}^{\alpha}h_{jj}^{\alpha}\right)\leq\dfrac{(n+2)p(n-p)-n^2}{(n-2)}\left((n-1)(c+\abs{H}^2)-\dfrac{Ric_{(p)}}{p}\right)-p(n-p)\abs{H}^2.
\end{align*}
Consequently, if
\begin{align*}
\dfrac{(n+2)p(n-p)-n^2}{(n-2)}\left((n-1)(c+\abs{H}^2)-\dfrac{Ric_{(p)}}{p}\right)-p(n-p)\abs{H}^2<p(n-p)c,
\end{align*}
or equivalently,
\begin{align*}
\dfrac{Ric_{(p)}}{p}>\left(n-1-\dfrac{(n-2)p(n-p)}{(n+2)p(n-p)-n^2}\right)\left(c+\abs{H}^2\right),
\end{align*}
we have
\begin{align*}
\sum_{i=1}^p\sum_{j=p+1}^n\sum_{\alpha=1}^m\left(2\left(h_{ij}^{\alpha}\right)^2-h_{ii}^{\alpha}h_{jj}^{\alpha}\right)<p(n-p)c.
\end{align*}

Finally, since $n\geq4$ and $c+\abs{H}^2\geq0$, we have
\begin{align*}
\dfrac{Ric_{(p)}}{p}\geq Ric_{\min},\\
(n-2)\left(c+\abs{H}^2\right)\geq\max_{1<p<n-1}\left(n-1-\dfrac{(n-2)p(n-p)}{(n+2)p(n-p)-n^2}\right)\left(c+\abs{H}^2\right),\\
(n-2)\left(c+\abs{H}^2\right)\geq\dfrac{n(n-1)}{n+2}\left(c+\abs{H}^2\right).
\end{align*}
Thus, if $Ric>(n-2)\left(c+\abs{H}^2\right)$ and $n\geq4$, then there is no stable integral $p$-currents for $0<p<n$.
\end{proof}

The idea of proving \autoref{thm:eigenricci} is
similar as the proof of \autoref{eigenest}. But since the Ricci curvature is the sum of sectional curvatures,
we must confront a more complicated algebra than the proof of theorem 1.1.
Hence, before proving  \autoref{thm:eigenricci}, we need an algebraic lemma.

Given a matrix $A\in M_{n\times n}(\R)$, we extend $A$ linearly into an operator $A:\Lambda^*\R^n\To \Lambda^*\R^n$ satisfying
\begin{align*}
A(\omega\wedge\eta)=A(\omega)\wedge\eta+\omega\wedge A(\eta),\quad\forall\omega, \eta\in \Lambda^*\R^n.
\end{align*}
For the matrix $A$, as an operator of $\R^n$ to itself, we denote by $\abs{A}_2$ its operator norm, i.e.,
\begin{align*}
\abs{A}_2\coloneqq\max_{0\neq x\in\R^n}\dfrac{\abs{Ax}}{\abs{x}}.
\end{align*}
It is obvious that $\abs{A}_2^2$ is the largest eigenvalue of $A^*A$.

Now we state the following algebraic Lemma.
\begin{lem}\label{lem:algebraic-eriji}For every symmetric matrices $A^{\alpha}\in M_{n\times n}(\R), 1\leq\alpha\leq m$, we have
\begin{align*}
\sum_{\alpha=1}^m\abs{A^{\alpha}\omega}^2\leq p^2\abs{\sum_{\alpha=1}^m\left(A_{\alpha}\right)^2}_{2}\abs{\omega}^2,\quad\forall \omega\in\Lambda^p\R^n, \quad 0\leq p\leq n.
\end{align*}
\end{lem}
\begin{proof} If $p=1$, a direct verification claims that this conjecture is true, i.e., for every real numbers $x_1,\dotsc, x_n$, we have
\begin{align*}
\sum_{\alpha=1}^m\sum_{i=1}^n\abs{\sum_{j=1}^nA^{\alpha}_{ij}x_j}^2\leq\abs{\sum_{\alpha=1}^m(A^{\alpha})^2}_2\sum_{j=1}^n\abs{x_j}^2.
\end{align*}

Set
\begin{align*}
\omega=\dfrac{1}{p!}\omega_{i_1\dotsc i_p}e_{i_1}\wedge\dotsm\wedge e_{i_p}=\sum_{1\leq i_1<\dotsm<i_p\leq n}\omega_{i_1\dotsc i_p}e_{i_1}\wedge\dotsm\wedge e_{i_p},
\end{align*}
then
\begin{align*}
A^{\alpha}\omega=&\dfrac{1}{p!}\omega_{i_1\dotsc i_p}A^{\alpha}(e_{i_1})\wedge\dotsm\wedge e_{i_p}+\dotsm+\dfrac{1}{p!}\omega_{i_1\dotsc i_p}e_{i_1}\wedge\dotsm\wedge A^{\alpha}(e_{i_p})\\
=&\dfrac{1}{(p-1)!}\omega_{i_1i_2\dotsc i_p}A^{\alpha}_{i_1k}e_k\wedge e_{i_2}\dotsm\wedge e_{i_p}\\
=&p\sum_{1\leq k<i_2<\dotsm<i_p\leq n}\tilde\omega_{k i_2\dotsc i_p}e_k\wedge e_{i_2}\dotsm\wedge e_{i_p},
\end{align*}
where $\tilde\omega_{k i_2\dotsc i_p}$ is the antisymmetrizer of $\sum_{i_1=1}^n\omega_{i_1i_2\dotsc i_p}A^{\alpha}_{i_1k}$, i.e.,
\begin{align*}
\tilde\omega_{k i_2\dotsc i_p}=\dfrac1p\left(\sum_{i_1=1}^n\omega_{i_1i_2\dotsc i_p}A^{\alpha}_{i_1k}-\sum_{i_1=1}^n\omega_{i_1i_2\dotsc i_{p-1}k}A^{\alpha}_{i_1i_p}-\dotsm-\sum_{i_1=1}^n\omega_{i_1ki_3\dotsc i_p}A^{\alpha}_{i_1i_2}\right).
\end{align*}
Therefore, we obtain
\begin{align*}
\sum_{\alpha=1}^m\abs{A^{\alpha}\omega}^2\leq&\dfrac{ p^2}{p!}\sum_{\alpha=1}^m\sum_{i_2,\dotsc,i_p,k}\abs{\tilde\omega_{ki_2\dotsc i_p}}^2\\
\leq&\dfrac{ p}{p!}\sum_{\alpha=1}^m\sum_{\#\set{i_2,\dotsc,i_p,k}=p}\left(\abs{\sum_{i_1}\omega_{i_1i_2\dotsc i_p}A^{\alpha}_{i_1k}}^2+\abs{\sum_{i_1}\omega_{i_1i_2\dotsc i_{p-1}k}A^{\alpha}_{i_1i_p}}^2+\dotsm+\abs{\sum_{i_1}\omega_{i_1ki_3\dotsc i_p}A^{\alpha}_{i_1i_2}}^2\right)\\
=&\dfrac{ p^2}{p!}\sum_{i_2,\dotsc,i_p}\sum_{\alpha=1}^m\sum_{k\notin\set{i_2,\dotsc,i_p}}\abs{\sum_{i_1}\omega_{i_1i_2\dotsc i_p}A^{\alpha}_{i_1k}}^2\\
\leq&\dfrac{ p^2}{p!}\sum_{i_2,\dotsc,i_p}\abs{\sum_{\alpha=1}^m(A^{\alpha})^2}_2\sum_{i_1}\omega_{i_1i_2\dotsc i_p}^2\\
=&p^2\abs{\sum_{\alpha=1}^m(A^{\alpha})^2}_2\abs{\omega}^2.
\end{align*}
\end{proof}

\begin{proof}[Proof of  \autoref{thm:eigenricci}]
Let $\omega\in\Omega^p(M)$ with $\abs{\omega}=1$.
First,
\begin{align*}
&\abs{\mathring{S}(\omega)-\dfrac{n-2p}{2}H\omega}^2-\dfrac{n^2}{4}\abs{H}^2\\
=&\abs{\mathring{S}(\omega)}^2-(n-2p)\hin{\mathring{S}(\omega)}{H\omega}-p(n-p)\abs{H}^2\\
=&\abs{\mathring{S}(\omega)-\dfrac{n-2}{2}pH\omega}^2+(p-1)n\hin{\mathring{S}(\omega)}{H\omega}-\dfrac{(n-2)^2}{4}p^2\abs{H}^2-p(n-p)\abs{H}^2.
\end{align*}

Moreover, a direct calculation yields
\begin{align*}
&\abs{\mathring{S}(\omega)-\dfrac{n-2p}{2}H\omega}^2-\dfrac{n^2}{4}\abs{H}^2\\
=&\abs{\mathring{S}(*\omega)-\dfrac{n-2(n-p)}{2}H*\omega}^2-\dfrac{n^2}{4}\abs{H}^2\\
=&\abs{\mathring{S}(*\omega)-\dfrac{n-2}{2}(n-p)H*\omega}^2+(n-p-1)n\hin{\mathring{S}(*\omega)}{H*\omega}\\
&-\dfrac{(n-2)^2}{4}(n-p)^2\abs{H}^2-p(n-p)\abs{H}^2\\
=&\abs{\mathring{S}(\omega)+\dfrac{n-2}{2}(n-p)H\omega}^2-(n-p-1)n\hin{\mathring{S}(\omega)}{H\omega}\\
&-\dfrac{(n-2)^2}{4}(n-p)^2\abs{H}^2-p(n-p)\abs{H}^2.
\end{align*}
Therefore,
\begin{align*}
&\abs{\mathring{S}(\omega)-\dfrac{n-2p}{2}H\omega}^2-\dfrac{n^2}{4}\abs{H}^2\\\\
=&\dfrac{n-p-1}{n-2}\left(\abs{\mathring{S}(\omega)-\dfrac{n-2}{2}pH\omega}^2-\dfrac{(n-2)^2}{4}p^2\abs{H}^2\right)\\
&+\dfrac{p-1}{n-2}\left(\abs{\mathring{S}(\omega)+\dfrac{n-2}{2}(n-p)H\omega}^2-\dfrac{(n-2)^2}{4}(n-p)^2\abs{H}^2\right)\\
&-p(n-p)\abs{H}^2.
\end{align*}

Notice that acting on $p$-forms, $\mathring{S}-\dfrac{n-2}{2}pH$ can be viewed as a linearly extension operator of $\left(\mathring{A}^{\alpha}-\dfrac{n-2}{2}H^{\alpha}\Id\right)\otimes\nu_{\alpha}$. In particular, applying \autoref{lem:algebraic-eriji}, we obtain
\begin{align*}
\abs{\mathring{S}(\omega)-\dfrac{n-2}{2}pH\omega}^2\leq p^2\abs{\sum_{\alpha=1}^m\left(\mathring{A}^{\alpha}-\dfrac{n-2}{2}H^{\alpha}\Id\right)^2}_2.
\end{align*}
Similarly,
\begin{align*}
\abs{\mathring{S}(\omega)+\dfrac{n-2}{2}(n-p)H\omega}^2=\abs{\mathring{S}(*\omega)-\dfrac{n-2}{2}(n-p)H*\omega}^2\leq (n-p)^2\abs{\sum_{\alpha=1}^m\left(\mathring{A}^{\alpha}-\dfrac{n-2}{2}H^{\alpha}\Id\right)^2}_2.
\end{align*}
Consequently,
\begin{align*}
&\abs{\mathring{S}\vert_{\Lambda^p}-\dfrac{n-2p}{2}H}_{op}^2-\dfrac{n^2}{4}\abs{H}^2\\
\leq&\dfrac{(n-p-1)p^2+(p-1)(n-p)^2}{n-2}\left(\abs{\sum_{\alpha=1}^m\left(\mathring{A}^{\alpha}-\dfrac{n-2}{2}H^{\alpha}\Id\right)^2}_2-\dfrac{(n-2)^2}{4}\abs{H}^2\right)\\
=&\dfrac{(n+2)p(n-p)-n^2}{n-2}\left(\abs{\sum_{\alpha=1}^m\left(\mathring{A}^{\alpha}-\dfrac{n-2}{2}H^{\alpha}\Id\right)^2}_2-\dfrac{(n-2)^2}{4}\abs{H}^2\right).
\end{align*}

By Gauss equation, we know that
\begin{align*}
Ric_{ii}=&\sum_{j=1}^n\bar R_{ijij}-\abs{\mathring{B}_{ii}}^2+(n-2)\hin{\mathring{B}_{ii}}{H}-\sum_{j\neq i}\abs{\mathring{B}_{ij}}^2+(n-1)\abs{H}^2\\
=&\sum_{j=1}^n\bar R_{ijij}-\abs{\mathring{B}_{ii}-\dfrac{n-2}{2}H}^2-\sum_{j\neq i}\abs{\mathring{B}_{ij}}^2+\dfrac{n^2}{4}\abs{H}^2.
\end{align*}
By choosing $\set{e_i}$ so that
\begin{align*}
\sum_{\alpha=1}^m\left(\mathring{A}^{\alpha}-\dfrac{n-2}{2}H^{\alpha}\Id\right)^2
\end{align*}
is a diagonalizing matrix, without loss of generality, assume the largest eigenvalue is $\abs{\mathring{B}_{11}-\dfrac{n-2}{2}H}^2$ we obtain
\begin{align*}
Ric_{\min}\leq Ric_{11}=&\sum_{j=1}^n\bar R_{1j1j}-\abs{\mathring{B}_{11}-\dfrac{n-2}{2}H}^2+\dfrac{n^2}{4}\abs{H}^2\\
=&\sum_{j=1}^n\bar R_{1j1j}-\abs{\sum_{\alpha=1}^m\left(\mathring{A}^{\alpha}-\dfrac{n-2}{2}H^{\alpha}\Id\right)^2}_2+\dfrac{n^2}{4}\abs{H}^2\\
\leq&(n-1)c^*-\abs{\sum_{\alpha=1}^m\left(\mathring{A}^{\alpha}-\dfrac{n-2}{2}H^{\alpha}\Id\right)^2}_2+\dfrac{n^2}{4}\abs{H}^2.
\end{align*}

As a consequence, we obtain
\begin{align*}
\abs{\mathring{S}\vert_{\Lambda^p}-\dfrac{n-2p}{2}H}_{op}^2-\dfrac{n^2}{4}\abs{H}^2\leq&\dfrac{(n+2)p(n-p)-n^2}{n-2}\left((n-1)(c^*+\abs{H}^2)-Ric_{\min}\right)-p(n-p)\abs{H}^2.
\end{align*}
Hence, by assumption
\begin{align*}
W^{[p]}\geq& p(n-p)c_*-\dfrac{(n+2)p(n-p)-n^2}{n-2}\left((n-1)(c^*+\abs{H}^2)-Ric_{\min}\right)+p(n-p)\abs{H}^2\\
=&\dfrac{(n+2)p(n-p)-n^2}{n-2}\left(Ric_{\min}-(n-1)\left(c^*+\abs{H}^2\right)+\dfrac{(n-2)p(n-p)}{(n+2)p(n-p)-n^2}\left(c_*+\abs{H}^2\right)\right).
\end{align*}
Therefore, according to \autoref{lem:Wp}, we have our conclusion.
\end{proof}

\begin{proof}[Proof of of weak Ejiri type \autoref{thm:eriji}] By \autoref{thm:eigenricci} and a similar argument as the proof of \autoref{rationalsphere1}.
\end{proof}

\appendix
\section{An Example}%\label{appendix}
The following example shows that the conditions mentioned in this paper are sharp.
\begin{eg}Consider the following Clifford torus
\begin{align*}\Lg{S}^p\left(\dfrac{\mu}{\sqrt{1+\mu^2}}\right)\times\Lg{S}^{n-p}\left(\dfrac{1}{\sqrt{1+\mu^2}}\right)\subset\Lg{S}^{n+1}(1)\subset\R^{n+2},\quad p=1,2,\dotsc, n-1,\quad \mu>0.
\end{align*}
It is obvious that $b_p\geq\max\set{p,n-p}$ and $\lambda_{1,p}=\lambda_{1,n-p}=0$.

Let $\phi=(x,y)$ be the position vector, then the first fundamental form is given by
\begin{align*}
\dif s^2=\dif x\dif x+\dif y\dif y.
\end{align*}
A unit norm vector field is $\nu=(-\mu^{-1}x,\mu y)$. Hence, the second fundamental form $B$ is
\begin{align*}
B=&-\hin{\dif(x,y)}{\dif(-\mu^{-1}x,\mu y)}\\
=&\mu^{-1}\dif x\dif x-\mu\dif y\dif y.
\end{align*}
Consequently the principal curvatures are $\mu^{-1}$ and $-\mu$ with multiplicity $p$ and $n-p$ respectively. In particular,
\begin{align*}
H=\dfrac{1}{n}\left(p\mu^{-1}-(n-p)\mu\right),\\
\abs{B}^2=p\mu^{-2}+(n-p)\mu^2,\\
\abs{\mathring{B}}^2=\dfrac{p(n-p)}{n}\left(\mu^{-1}+\mu\right)^2.
\end{align*}
Moreover, the sectional curvature satisfies
\begin{equation*}
K_{ij}=
\begin{cases}
1+\mu^{-2},& 1\leq i< j\leq p;\\
1+\mu^{2},&p+1\leq i\neq j\leq n;\\
0,&1\leq i\leq p, p+1\leq j\leq n.
\end{cases}
\end{equation*}
The Ricci curvature satisfies
\begin{equation*}
Ric_{ii}=
\begin{cases}
(p-1)(1+\mu^{-2}),&1\leq i\leq p;\\
(n-p-1)(1+\mu^2),&p+1\leq i\leq n.
\end{cases}
\end{equation*}

 A direct computation gives the following: if $(n-2p)(p\mu^{-1}-(n-p)\mu)\leq0$, then
 \begin{align*}
&\dfrac{\abs{\mathring{B}}}{\sqrt{n}}+\dfrac{\abs{n-2p}\abs{H}}{2\sqrt{p(n-p)}}-\sqrt{1+\dfrac{n^2\abs{H}^2}{4p(n-p)}}\\
=&\dfrac{\sqrt{p(n-p)}(\mu^{-1}+\mu)}{n}+\dfrac{\abs{n-2p}\abs{p\mu^{-1}-(n-p)\mu}}{2n\sqrt{p(n-p)}}-\dfrac{(p\mu^{-1}+(n-p)\mu)}{2\sqrt{p(n-p)}}\\
=&\dfrac{2p(n-p)(\mu^{-1}+\mu)+\abs{n-2p}\abs{p\mu^{-1}-(n-p)\mu}-(np\mu^{-1}+n(n-p)\mu)}{2n\sqrt{p(n-p)}}\\
=&\dfrac{2p(n-p)(\mu^{-1}+\mu)-(n-2p)(p\mu^{-1}-(n-p)\mu)-(np\mu^{-1}+n(n-p)\mu)}{2n\sqrt{p(n-p)}}\\
=&0.
\end{align*}
Hence, if $(n-2p)(p\mu^{-1}-(n-p)\mu)\leq0$, we obtain that $\abs{B}^2=\alpha(1,p,n,H)$.

When $p=1$ or $p=n-1$, we have
$$
Ric_{ min}-\left(n-1 -\frac{(n-2)p(n-p)}{(n+2)p(n-p)-n^2}\right)(1+\vert H\vert^2)
=0.
$$
When $1< p<n-1$, taking $\mu = \sqrt{\frac{p-1}{n-p-1}}$, we have $Ric_{ii}\equiv n-2$ for all $1\le i \le n$ which implies $Ric= (n-2)g$.
Therefore,
\begin{align*}
  &Ric_{min} -\left(n-1 -\frac{(n-2)p(n-p)}{(n+2)p(n-p)-n^2}\right)(1+\vert H\vert^2)\\
  &=Ric_{min} -\frac{n^2(p-1)(n-p-1)}{(n+2)p(n-p)-n^2}\cdot\left(1+\left(\frac{p\mu^{-1} -(n-p)\mu}{n}\right)^2\right)\\
  &=n-2-\frac{n^2(p-1)(n-p-1)}{(n+2)p(n-p)-n^2}\cdot\frac{(n-2)((n+2)p(n-p)-n^2)}{n^2(p-1)(n-p-1)}\\
  &=n-2-(n-2)\\
  &=0.
\end{align*}
\end{eg}

%\bibliographystyle{amsplain}
 %\bibliography{library}

\end{document}